\documentclass[10pt]{amsart}
\usepackage{t1enc}
\usepackage{amsfonts,amsmath,amssymb,amsthm}

\newtheorem{theorem}{Theorem}[section]
\newtheorem{lemma}[theorem]{Lemma}
\newtheorem{proposition}[theorem]{Proposition}
\newtheorem{corollary}[theorem]{Corollary}
\theoremstyle{definition}
\newtheorem*{definition}{Definition}
\theoremstyle{remark}

\numberwithin{equation}{section}

\newcommand{\CC}{\mathbb C}

\newcommand{\RR}{{\mathbb  R}}
\newcommand{\BB}{{\mathbf  B}}
\newcommand{\RRplus}{{\RR}_+} 
\newcommand{\RRn}{{\RR}^n}
\newcommand{\bS}{{\mathbb S}}
\newcommand{\prn}{\RRn \times \RRn}

\newcommand{\ep}{\varepsilon}
\newcommand{\hake}[1]{\langle #1 \rangle }
\newcommand{\scalar}[2]{\langle #1, #2 \rangle}

\newcommand{\norm}[1]{\Vert #1 \Vert }
\newcommand{\normrum}[2]{{\norm {#1}}_{#2}}
\newcommand{\upp}[1]{^{(#1)}}
\newcommand{\p}{\partial}
\newcommand{\ee}{\text{\rm e}}
\newcommand{\ii}{\text{\rm i}}
\newcommand{\dd }{\,\text{\rm d}}
\newcommand{\opn}{\operatorname}

\newcommand{\Cal}{\mathcal}

\title{Local smoothing for the backscattering transform}

\author[Ingrid Belti\c t\u a]{Ingrid Belti\c t\u a $^*$}
\thanks{$^*$ Partially supported by 
SPECT Short Visit Grant 1006 and the grant 2-CEx05-11-23/2005}
\address{Institute of Mathematics "Simion Stoilow"
of the Romanian Academy, PO Box 1-764, RO 014700 Bucharest, Romania}
\email{Ingrid.Beltita@imar.ro} 
\author[Anders Melin]{Anders Melin}
\address{Lund University,  Sweden}
\email{andersmelin@hotmail.com}

\begin{document}
\parskip=6pt
\baselineskip=16.5pt
\bigskip

\maketitle
\begin{abstract}
 An analysis  of the backscattering data for the Schr\"odinger 
 operator in odd dimensions $n\ge 3$ 
 motivates the introduction of the backscattering transform  
 $B: C_0^\infty ({\mathbb R}^n;{\mathbb C} )\to
C^\infty ({\mathbb R}^n; {\mathbb C})$. 
This is an entire analytic mapping 
and we write
$
Bv = \sum _1^\infty B_Nv
$
where $B_Nv$ is the $N$:th order term in  the power series  
expansion at $v=0$. In this paper we study
 estimates for $B_Nv$  in $H_{(s)}$ spaces, 
and prove that $Bv$ is entire
analytic in $v \in H_{(s)}\cap \Cal E'$ when $s\ge (n-3)/2$.
\end{abstract}

\section{Introduction}

The present note is devoted to proving
continuity and smoothing properties of the  backscattering transform
for the Schr\"odinger operator in odd dimensions $n >1$.

In order to state the main result a brief description of the
mathematical objects involved is necessary. (The reader is referred to 
\cite{M:contemp}, \cite{M:rims}, \cite{AM99} for details.) 

Consider the Schr\"odinger operator $H_v =-\Delta +v$ in $\RRn$, where
$v \in L^2_{\opn{cpt}} (\RR^n)$. Assume that $H_v$ with domain $H_{(2)}(\RR^n)$ is
self-adjoint and the wave operators
$$
W_{\pm} = \lim _{t\to \pm \infty}  e^{\ii tH_v}e^{-\ii tH_0}
$$
exist. 
Then the operator $v W_+$ is continuous
from $L^2$ to $L^1$, and therefore its distribution kernel 
$ v(x) W_{+} (x, y)$ is defined. 
After composing 
it
with a non-singular linear transformation,   we arrive at
the distribution $v(x-y)W_+(x-y,x+y)$ in $\Cal D'(\prn )$. Since $v$ is
compactly supported we may integrate with respect to
 $y$ and  obtain   the distribution
$$
2^n \int v(x-y) W_+(x-y,x+y) \dd y.
$$
(The normalization factor here is introduced in order that the
expression above  should be equal to $v(x)$ when $W_+$ is replaced by
the identity.)
It was proved in \cite{M:contemp} that when $v\in C_0^\infty(\RR^n; \RR)$ the integral 
above represents 
the inverse Fourier transform of the  backscattering part of the
scattering  matrix, when this is represented as a function in the
momentum variables. 
The real part of the expression above  is equal to
$$
{\beta}v(x)=2^n \int v(x-y)W(x-y,x+y) \dd y,
$$
where the operator $W=(W_++W_-)/2$ has a real-valued distribution
kernel.

The  backscattering transform $Bv $ of $v\in C_0^\infty(\RR^n; \RR)$ is
a slight  modification of  ${\beta}v$. 
Let $K_v(t)$ be the wave group associated to the operator
$$
\Box_v =\p_t^2 -\Delta_x + v,
$$
i.e.,  $u(x, t)= (K_v(t) f)(x)$ is, for every $f \in C_0^\infty(\RRn)$, 
the unique solution in $C^1([0, \infty), L^2(\RRn))$ to the
Cauchy problem
$$ \Box_v u(x, t)=0, \quad u(x, 0) = 0, \quad (\p_t u)(x, 0) = f(x).
$$ 
Then $K_v(t)$ is a strongly continuous function of $t$ with values in 
the space of bounded linear operators on $L^2(\RRn)$.
(See \cite{M:rims} for details.) 
We have that $|x-y|\le t$ in the support of $K_v(x, y; t)$ and 
$|x-y|=t$ in  the support of $K_0(x, y; t)$. 
This ensures that the operator 
$$ G=- \int\limits_0^\infty K_v(t) v \dot{K}_0(t) \dd t$$
is well-defined and continuous on $L^2_{\opn{cpt}}(\RRn)$, where the dot denotes
derivative in the variable $t$.
Theorem~7.1 in \cite{M:contemp} gives the relation between $G$ and $W$ above:
There exist an orthonormal basis
$(f_j)_{1\leq j \leq \mu} $  of real eigenfunctions corresponding to
the  negative part of the spectrum of $H_v$ and a set $(g_j)_{1\leq j
  \leq {\mu}}$ of smooth real-valued functions such that 
$$
W= I +G + \sum _1 ^{\mu} f_j \otimes g_j.
$$
It turns out (see below) that $G=G_v$, considered as function of $v$ with values in the space of 
continuous linear
operators in $L_{\opn{cpt}}^2 (\RR^n)$, extends to an entire analytic function of 
$v\in C_0^\infty(\RR^n)$, i.e., to the space of complex-valued $v$ in $C_0^\infty$.
Also, if $v$ is sufficiently small (in a sense that we do not make precise here),
there are no bound states and  $W=I+G$ then.
For these reasons it is natural to modify the definition of $\beta v$ by subtracting 
the contribution from 
$\sum _1 ^{\mu} f_j \otimes g_j$.

\begin{definition}
Assume $v\in C_0^\infty (\RRn; \CC)$.
The backscattering transform $Bv$ of $v$ is defined
by
$$(Bv)(x) = v(x) + 2^n \int\limits v(x-y) G(x-y, x+y)\dd y. $$
Here the integral is taken in distribution sense and $v(x)G(x, y)$ is the distribution 
kernel of the operator $v G$.
\end{definition}   

 It was proved in \cite{M:rims} that $G=G_v$  extends to an  entire analytic  function of $v \in
 L_{\opn{cpt} }^q(\RRn)$
when $q>n$.  For such $v$ we  can define $Bv$ again as in the previous  definition and
 $Bv$  will be entire analytic in $v$ with values in $\Cal D'(\RRn)$. 
 We write
$$
Bv = \sum _1^\infty B_Nv
$$
where $B_Nv$ is the $N$:th order term in  the power series  expansion at $v=0$.
There are  other spaces of $v$  (containing $C_0^\infty$ as a dense subset)  to which
 $Bv$  can be extended analytically. 
For reasons of continuity such expansions  can be studied by deriving
 estimates for the $B_Nv$ when $v\in C_0^\infty$.  
In this paper we shall study
 estimates for $B_Nv$  in $H_{(s)}$ spaces, 
and prove that $Bv$ is entire
analytic in $v \in H_{(s)}\cap \Cal E'$ when $s\ge (n-3)/2$.

We recall some basic ingredients in the construction of $Bv$  when $v\in C_0^\infty(\RR^n)$.
We recall from \cite{M:rims}, or section 11 in \cite{M:contemp},
that 
\begin{equation}\label{bn:1}
K_v(t) =\sum\limits_{N\ge 0} (-1)^N K_N(t),
\end{equation}
where $K_N $
are inductively defined by
\begin{equation}\label{bn:2}
\begin{gathered}
K_0(t) = \frac{\sin t|D|}{|D|},
\\
K_N(t) = ( K_{N-1} \ast v K_0)(t) =\int\limits_0^t K_{N-1}(s) v K_0(t-s) \dd s, \quad N\ge 1.
\end{gathered}
\end{equation}
One has the estimate
$$
\normrum{K_N(t)}{L^2\to L^2} \leq \normrum v{L^\infty}^N t^{2N+1}/(2N+1)!
$$
Since the distribution kernel  $K_N(x,y;t)$ of $K_N(t)$ is supported
in the set where $|x-y|\leq t$, it makes sense to consider 
\begin{equation}\label{bn:4}
G_N =(-1)^{N} \int _0 ^\infty K_{N-1}(t)v\dot K _0(t) \dd t.
\end{equation}
This is a continuous linear operator in $L_{\opn{cpt}}^2(\RRn)$, 
and the estimates for the $K_N$ show that
$$
G = \sum _1^\infty  G_N
$$
is an entire
analytic  function of $v$. 
We see that
\begin{equation}\label{bn:5}
(B_{N} v)(x) = 2^{n}\int v(y) G_{N-1}(y, 2x-y) \dd y, \quad N\ge 2.
\end{equation}

The following theorem (Theorem~8, \cite{M:rims}) 
reveals the smoothing properties of $B_N$ for large $N$.

\begin{theorem}\label{cor:rims}
Let $q>n$ and $k$ be a nonnegative integer.
Then there is a positive integer $N_0=N_0(n, q, k)$ such that
$\Delta^k B_N v \in L^2_{\opn{loc}}(\RRn)$ when $v\in L^q(\RRn)$ has compact support and $N\ge N_0$.
Moreover, if $R_1, R_2> 0$, there is a constant $C$, depending on $n$, $k$, $R_1$, $R_2$ 
and $q$ only such that
$$ \normrum{ \Delta^k B_N v }{L^2(B(0, R_1))} \le C^N \normrum{v}{L^q}^N/N!, \qquad N\ge N_0,$$
whenever $v\in L^q(\RRn)$ has support in the ball $B(0, R_2)$.
\end{theorem}

The aim of this paper is to study (local) continuity properties of
the operators $B_N$ in $H_{(s)}$ spaces.

Let $\normrum{\cdot}{(s)}$ denote the  norm on
the Sobolev space $H_{(s)}(\RRn)$.
Also $H_{(s)}(\Omega)$, $s\ge 0$, is 
the space of functions which are restrictions to $\Omega$ of 
functions from the Sobolev space $H_{(s)}(\RRn)$, when $\Omega$ is an open
set with smooth boundary. 
The norm on $H_{(s)}(\Omega)$, $s\ge 0$, is the quotient norm
$$ \normrum{f}{H_{(s)}(\Omega)} =\inf\{ \normrum{F}{(s)}; F \in H_{(s)}(\RRn), F= f \, \text{in}\, \Omega\}.$$ 

Our main result here is contained in the next theorem.

\begin{theorem}\label{mainthm:1}
Assume  $0\le a \le s-(n-3)/2$, and let N$(a, s)$ be the smallest integer $N$ such that
$a<N-1$ and $a\le (N-1) (s-(n-3)/2)$.
Then  there is a constant $C$, which depends on $n$, $s$ and  $a$ only,
such that
$$ 
\normrum{B_N v}{H_{(s+a)}(B(0, R))} \le C^N R^{(N-1)/2} N^{-N/2} \normrum{v}{(s)}^N 
$$
when $N\ge N(a, s)$, $R>0$ and  
$v \in C_0^\infty(B(0, R))$. 
 \end{theorem}

A first corollary of this result is the above-mentioned analyticity of
the backscattering transformation.

\begin{corollary}\label{anal}
The mapping $C_0^\infty(\RR^n) \ni v \to Bv \in  C^\infty(\RR^n)$ extends to 
an entire analytic mapping from 
$ H_{(s)}(\RR^n)\cap {\Cal E}'(\RR^n)$ to  $H_{(s), \opn{loc}}(\RR^n) $
whenever $s\ge (n-3)/2$.
\end{corollary}

A second corollary gives the regularity of the difference between $v$ 
and its backscattering transform.

\begin{corollary}\label{cor:reg}
Assume $s\ge (n-3)/2$ and $0\le a <1$ satisfy
$a\le s-(n-3)/2$.
If $v\in H_{(s)}(\RR^n)$ is compactly  supported, then
\begin{equation}\label{1} 
v-Bv \in H_{(s+a), \opn{loc}}(\RR^n).
\end{equation}
\end{corollary}

The outline of this note is as follows.
In the  next section we derive a formula that generalizes to arbitrary $N>2$ the formula
$$
 (B_2v)(x) =  
\int\limits_{(\RR^n)^{2}}
E_{2}(y_1, y_{2})
 v(x-\frac{y_2-y_1}{2}) v(x-\frac{y_1+y_2}{2})
\dd y_1 \dd y_2, 
$$
which appears in Corollary~10.7 of \cite{M:contemp}. 
Here $E_2$ is the unique fundamental solution of the ultra-hyperbolic
operator $\Delta_x-\Delta_y$ such that $E_2(x, y)= - E_2(y, x)$ and
$E_2$ is rotation invariant 
separately in $x$ and $y$. 
When $N>2$ we have to replace $E_2$  by a
distribution $E_N \in \Cal D'((\RRn )^N)$ which is a fundamental
solution of the operator 
$
P_N =(\Delta _{x_N}-\Delta _{x_1})(\Delta _{x_N}-\Delta _{x_2})\cdots
(\Delta _{x_N}-\Delta _{x_{N-1}})
$.
The distribution $E_N$ is discussed in more detail in Section~3.

Once these formulas have been obtained, the proof of the theorem becomes elementary.
The third section contains estimates of the Fourier transforms of  (cut-offs of) $E_N$.
These are in turn used in the fourth section when the estimates in 
Theorem~\ref{mainthm:1} are obtained by Fourier transforming the formula for $B_N v$.

We close this presentation with a few words on the existing literature on backscattering 
problems for the potential scattering in odd dimensions.
The backscattering map was studied also in \cite{ER89} for dimension $3$ and in \cite{ER92}
for arbitrary dimensions, and local uniqueness was proved 
for potentials in a certain weighted H\"older space.
The actual backscattering transform defined as above was considered in
\cite{L} for dimension $3$, and it was proved to be analytic 
when defined  on small potentials $v$ such that $\nabla v\in L^1$ and with values in the same space, and consequently uniqueness for the inverse backscattering problem was obtained for small potentials in this space.
Generic uniqueness was proved in \cite{S} for compactly supported bounded 
potentials in dimension $3$.
We also mention \cite{U} for an approach using  Lax-Phillips scattering.
The problem of recovering the singularities of $v$ from the backscattering data was
considered in \cite{GU}, \cite{J} and \cite{ruiz}.
Our result here improves the results in \cite{ruiz} in the sense that 
it shows that the  difference between the potential $v$ and its backscattering transform 
is more regular 
and the result holds for arbitrary odd $n\ge 3$.

Finally, let us fix some notation we use throughout the paper.
If $N\ge 2$ we use the notation $\vec x= (x_1, \dots, x_N)\in (\RR^m)^N$ where
$x_1, \dots x_N\in \RR^m$, for $m$ a positive integer. 
If $x\in \RR^m$ we shall set $\hake{x}=(1+|x|^2)^{1/2}$.
The Fourier transform of a distribution $u$ will be denoted either by
$\hat u$ or by ${\Cal F}u$.

\section{A formula for $B_N$}

In this section we are going to write $B_Nv$ as the value at $(v, \dots, v)$ of a $N$-linear operator defined 
from $C_0^\infty(\RR^n)\times \cdots \times C_0^\infty(\RR^n)$ to $C^\infty(\RRn)$, 
following the procedure in \cite{M:rims}.
The key point here is the fact that $K_0(t)$ obeys Huygens' principle, more specifically, 
that its convolution kernel $k_0(x;t)$ is supported in the set where $|x|=t$.

When $N=1, 2, \dots$ 
we define $Q_N \in {\Cal D}'((\RR^n)^N \times \RRplus)$ inductively
by
\begin{gather}
Q_1(x; t) =  k_0(x; t),\label{en:1}\\
Q_N(x_1, \dots, x_N; t) =\int\limits_0^t Q_{N-1}(x_1, \dots x_{N-1}; t-s) Q_1(x_N; s)\dd s  \quad \text{when } N\ge 2.\label{en:2}
\end{gather} 
Then the mapping
$$
\RRplus\ni t\rightarrow Q_N(x_1, \dots x_N; t) \in {\Cal D}'((\RR^n)^N)
$$
is smooth when $N\ge 1$. 
It is easily seen that $Q_N$ is symmetric in $x_1, \dots x_N$, rotation invariant separately in these variables and, since $|x|=t$ in the support of $k_0(x, t)$, it follows that 
\begin{equation}\label{en:3} 
|x_1|+\cdots +|x_N|= t \quad \text{in}\quad \opn{supp}Q_N.
\end{equation}
Next we define $E_N\in {\Cal D}'((\RR^n)^N)$, $N\ge 2$, by
\begin{equation}\label{en:5}
E_N(x_1, \dots, x_N) = (-1)^{N-1}\int\limits_0^\infty Q_{N-1}(x_1, \dots, x_{N-1}; t) \dot{k}_0(x_N; t) \dd t.
\end{equation} 
It follows from \eqref{en:3} that
\begin{equation}\label{en:6}
|x_1|+\cdots+ |x_{N-1}|= |x_N| \quad \text{in}\quad \opn{supp}E_N,
\end{equation}
$E_N$ is rotation invariant separately in all variables, and symmetric in 
$x_1, \dots, x_{N-1}$. 
We recall here that
$$
E_2(x, y) = 4^{-1} (\ii\pi)^{1-n} \delta\upp{n-2}(x^2-y^2) \quad \text{on}\quad  \RRn\times \RRn
$$
is the unique fundamental solution of the ultra-hyperbolic operator
$\Delta_x-\Delta_y$ such that $E_2(x, y)= - E_2(y, x)$ and $E_2$ is rotation invariant
separately in $x$ and $y$. (See Theorem~10.4 and Corollary~10.2 in 
\cite{M:contemp}.)

The next lemma follows easily from \eqref{bn:2} and \eqref{bn:4} 
by induction and some simple computations.
\begin{lemma}\label{lemma:bn1}
Assume $v\in C_0^\infty(\RR^n)$. Then
\begin{align}
 K_N(x, y; t)&  = \int v(x_1)\cdots v(x_N) Q_{N+1} 
(x-x_1, x_1-x_2, \dots, x_{N-1}-x_N,  x_N-y; t) \dd \vec x, \label{en:7001}\\
G_N(x, y) &  = \int\limits_{(\RR^n)^N} 
v(x_1) \cdots  v(x_N) E_{N+1} (x-x_1, x_1-x_2, \dots, x_{N-1}- x_{N},  x_N-y)  \dd \vec x \label{en:8001}
\end{align}
for every $N\ge 1$.
\end{lemma}

\begin{proposition}\label{prop:en1}
For $N\ge 2$
$$
(B_{N}v)(x) =  
\int\limits_{(\RR^n)^{N}}
E_{N}(y_1, \dots, y_{N})
 v(x-\frac{y_{N}}{2} - Y_0) v(x-\frac{y_{N}}{2} - Y_1)\cdots
v(x- \frac{y_{N}}{2} -Y_{N-1}) \dd \vec {y}
$$ when $v\in C_0^\infty(\RR^n)$, where
$$ 
Y_{0} =\frac{1}{2}\sum\limits_{j=1}^{N-1} y_j \quad  \text{and} \quad Y_k  =Y_0- \sum\limits_{j=1}^{k}y_j, \quad 1\le k\le N-1.
$$ 
\end{proposition} 
\begin{proof}
We use \eqref{en:8001} to express $G_{N-1}(y, 2x-y)$ in \eqref{bn:5} and  
get thus
$$
\begin{gathered}
(B_{N} v)(x) = 2^n \int\limits_{\RR^n\times (\RR^n)^{N-1}}
v(y) v(x_1) \cdots v(x_{N-1})\\
 E_{N} (y-x_1, x_1-x_2, \dots, x_{N-2}- x_{N-1},  x_{N-1} + y -2x) \dd y\dd \vec x.
\end{gathered}
$$
The proposition follows by changing variables 
$ y-x_1=-y_1$, $x_1-x_2=-y_2$, $\dots$, $x_{N-2}-x_{N-1}= - y_{N-1}$,  
$x_{N-1} + y -2x = -y_{N}$,
hence
\allowdisplaybreaks
\begin{gather}
y = x- \frac{1}{2} \sum\limits_{j=1}^{N} y_j = x-\frac{y_{N}}{2} -Y_0
\nonumber
\\
x_1 = y+y_1 =  x-\frac{y_{N}}{2} - Y_1
\nonumber
\\
\dots 
\nonumber
\\
x_{N-1} = x_{N-2} + y_{N-1} =  x-\frac{y_{N}}{2} - Y_{N-1}.
\nonumber
\end{gather}
Here we have made use of the invariance properties of $E_N$, 
which in particular ensure that $E_N(y_1,\dots, y_N)$ is even in each $y_j$.
\end{proof}

\section{The distribution $E_N$}
We need some further information on the distribution $E_N$ defined 
in \eqref{en:5}.

The first result is a characterization of $E_N$.
We denote
$$
P_N = (\Delta _1- \Delta _N)\cdots (\Delta _{N-1}- \Delta _N),
$$
where $\Delta _j$ in the Laplacian in the variables
$x_j$. 

\begin{lemma}\label{lemma:uniqueness}
The distribution $E_N$ is a fundamental solution of $P_N$. 
It has the  following properties:
\begin{itemize}
\item[\rm (i)] $E_N(x_1, \dots , x_N)$ is  rotation invariant in each $x_j$;
\item[\rm (ii)]
$|x_1|+ \cdots + |x_{N-1}|= |x_N| $ in the support of $E_N$;
\item[\rm (iii)]
$E_N$ is homogeneous of degree $2(N-1)-nN$.
\end{itemize}
If $E$ is a fundamental solution of $P_N$ that satisfies {\rm (i)-(iii)},
then $E=E_N$.
\end{lemma}

\begin{proof}
We first  prove that $P_NE_N = \delta (x_1, \dots ,x_N)$, and  when
doing this we may assume that $N \geq 3$. Since $ \p
_t ^2 k_0(x;t) = \Delta _x k_0(x;t)$, it follows easily from (2.2)
with $N$ replaced  by $N-1$ that 
$$
\begin{gathered}
\p _t ^2 Q_{N-1}(x_1, \dots , x_{N-1};t) = \Delta _{N-1} Q_{N-1}(x_1,
\dots , x_{N-1};t)\\
 +
Q_{N-2}(x_1, \dots , x_{N-2};t)\delta (x_{N-1}).
\end{gathered}
$$
It follows  from (2.4) then that
$$
\begin{gathered}
\Delta _N E_N(x_1, \dots , x_N) = (-1)^{N-1}\int\limits _0 ^\infty Q_{N-1}(x_1, \dots
, x_{N-1};t)\p _t ^2 \dot k _0(x_N;t) \dd t \\
=(-1)^{N-1} \int\limits _0 ^\infty (\p _t ^2 Q_{N-1}(x_1, \dots, x_{N-1};t))\dot
k_0(x_N;t)\dd t \\
=(-1)^{N-1} \Delta_{N-1}\int\limits _0 ^\infty Q_{N-1 }(x_1, \dots , x_{N-1};t) \dot
k_0(x_N;t) \dd t \\
+(-1)^{N-1} \int\limits _0 ^\infty Q_{N-2}(x_1, \dots ,x_{N-2};t)\delta
(x_{N-1}) \dot k_0(x_N;t) \dd t \\
=\Delta _{N-1} E_N(x_1, \dots , x_N) -E_{N-1}(x_1, \dots ,
x_{N-2},x_N)\delta (x_{N-1}).
\end{gathered}
$$
We have proved therefore that
\begin{equation}\label{extra:1}
(\Delta _{N-1}- \Delta _{N})E_N(x_1, \dots ,x_N) = E_{N-1}(x_1, \dots ,
x_{N-2},x_N) \delta (x_{N-1}).
\end{equation}
Assuming, as we may, that the assertion has been  proved for lower
values of $N$ and letting 
$(\Delta _1- \Delta _N) \cdots (\Delta _{N-2}-\Delta _N) 
$ act on both sides of \eqref{extra:1} we may conclude  that
$P_NE_N(x_1,  \dots , x_N)  = \delta (x_1, \dots x_N) $.

The  conditions (i) and (ii) are  simple consequences of the
definitions, together with  the fact that $k_0(x;t)$ is rotation invariant in $x$
and  supported in the set where $|x|=t$. 
Since $k_0$ is homogeneous
when considered as a distribution in $x$ and $t$, it follows  
that $E_N$ is a homogeneous  distribution. 
Its degree of homogeneity must  be equal to the degree of $P_N$  minus  the dimension of
$(\RRn)^N$. 
This proves (iii).

It remains to prove that ${\Phi}=0$ if ${\Phi}={\Phi}(x_1, \dots ,
x_N) $ is a distribution satisfying the conditions in (i) -(iii) and
$P_N {\Phi} =0$.

Define
$$
\Psi (x_1, \dots , x_N) = (\Delta _1 - \Delta _N)\cdots (\Delta
_{N-2}- \Delta_N) {\Phi}(x_1, \dots , x_N)
$$
(with the interpretation ${\Psi}=E_2$ if $N=2$).
This a homogeneous distribution of degree $2- nN$ and 
$$
(\Delta _{N-1} - \Delta _N) {\Psi} =0.
$$
Since ${\Psi}$ is rotation invariant in each $x_j$,  it follows from
Theorem 10.1 of \cite{M:contemp} that
${\Psi}$ is symmetric in  $x_{N-1}, x_N$.  
Since $|x_1| + \cdots
+|x_{N-1}| = |x_N| $ in the support of ${\Psi}$  this implies that
$x_1= \cdots = x_{N-2}=0 $ in its support. Hence 
$$
{\Psi}(x_1, \dots , x_N) = \sum {\delta} \upp {\alpha} (x_1, \dots
x_{N-2}) u_{\alpha}(x_{N-1}, x_N),
$$
where the $u_{\alpha}(x,y) \in \Cal D'(\prn )$ are solutions to the
ultra-hyperbolic equation. 
The  rotation invariance of
${\Phi}$ in the $x_j$ implies that the summation takes place over
even $|{\alpha}|$ only and that the $u_{\alpha}(x,y) 
$   and rotation invariant  separately in $x$
and $y$. 
Also, $u_{\alpha}(x,y) = u_{\alpha}(y,x) $ and $u_{\alpha}$
is homogeneous of degree ${\mu}_{\alpha}$, where
$$
{\mu}_{\alpha}= 2- nN +(N-2) n +|\alpha| = 2 +|\alpha|-2n 
$$
is even. Since ${\mu}_{\alpha}>-2n$ 
the proof is completed if we prove that
  $u_{\alpha} $ vanishes outside the origin in $\prn$. 
In this set we  may  view $u_{\alpha}$ as a function $f(s,t)$ in $s=|x|,
  t=|y|$. 
Since it is supported in the set where $s=t$ we may write
$$
f(s,t) = \sum _{0 \leq j \leq J} c_j \delta \upp j(s-t) (s+t)
^{j+{\nu}}
$$
where ${\nu} =1+{\mu}_{\alpha}$ is odd, and the  summation takes place
over even $j$ only, since $f(s,t)=f(t,s)$. 
We  assume that $f \neq 0$ and
shall see that this leads to a contradiction.

 Assume now  that $c_J
\neq 0$. 
Expressing  the Laplacian in polar coordinates, we get the
equation
$$
0 = \Big (\p _s ^2 -\p _t ^2 + (n-1) (s^{-1} \p _s -t^{-1}\p _t)\Big) f(s,t).
$$
The  right-hand side here is a linear combination of $\delta \upp j
(s-t) (s+t) ^{j+\nu -2}$ with $j \leq J+1$, and a simple computation
shows that the coefficient in front of $\delta \upp{J+1}(s+t)^{J+\nu
  -1}$ is equal to 
$4c_J {\kappa}$, where
$$
{\kappa} = (J+{\nu}) +n-1.
$$
This  gives us a contradiction,  since  we know that ${\kappa}=0$ while
the  right-hand side above is an odd integer.  We have proved
therefore that $u_{\alpha}$ vanishes outside  the origin. 
\end{proof}

We need to establish estimates for the Fourier transforms of certain
cut-offs of $E_N$.
Namely, we shall consider distributions of the form
\begin{equation}\label{newton:1}
(-1)^{N-1}\int\limits_0^\infty Q_{N-1}(x_1, \dots, x_{N-1}; t) \dot{k}_0(x_N; t)
\chi(t) \dd t, \qquad N=2, 3, \dots,
\end{equation}
where $\chi\in C_0^\infty(\RR)$.
We notice that $|x_1|+\cdots +|x_{N-1}|= |x_N|< R_0$ in the support of this 
distribution whenever the support of $\chi$ is contained in the interval $(-\infty, R_0)$.
Also, if $\chi(t)=1$ when $0\le t\le  R_1$, then the restrictions to 
$(\RR^n)^{N-1} \times B(0, R_1) $ 
of the distribution in \eqref{newton:1} and of $E_N$ coincide.

We start with some preparatory computations.
When $a\in \RR$ define
$$
\varphi _a (t) = Y_+ (t) \frac {\sin (ta)}a, \qquad t \in \RR,
$$ 
where  $Y_+$ is the Heaviside's function.

\begin{lemma}\label{lemmanewton:1}
Assume $N\ge 2$ and $a_1, \dots a_N$ are  
real numbers such that $a_j^2 \ne a_k^2$ when $j\ne k$.
Then we have the identity
\begin{equation}\label{newton:2}
(\varphi _{a_1}\ast \cdots \ast \varphi _{a_N})(t) = \sum _{j=1} ^N 
\prod _{k \neq j} \frac 1{a_k^2-a_j^2}\varphi _{a_j}(t).
\end{equation}
\end{lemma}

\begin{proof}
Let $\ep >0$ and define ${\psi}_j(t) = e^{-\ep t} \varphi _{a_j}(t)$.
A simple computation  shows that
$$
\widehat{{\psi}}_j({\tau}) = \frac 1{(\ep +\ii{\tau})^2+a_j^2}.
$$
If ${\Psi}= {\psi}_1 \ast \cdots \ast {\psi}_N$ it follows that
$$
\begin{gathered}
\widehat {\Psi}({\tau}) = \prod _1 ^N \frac 1{(\ep+\ii \tau)^2 +a_j^2}=\sum _{j=1}^N \Big ( \prod _{k\neq j}  \frac 1{a_k^2-a_j^2}\Big )
\frac 1{(\ep +i {\tau})^2 +a_j^2}\\
=  \sum _{j=1}^N \Big ( \prod _{k\neq j}  \frac 1{a_k^2-a_j^2}\Big
)\widehat  {\psi} _j({\tau})
\end{gathered}
$$
Hence 
$$
{\Psi} (t) = \sum _{j=1}^N\Big ( \prod _{k\neq j}  \frac 1{a_k^2-a_j^2}\Big
)  {\psi}_j (t).
$$
The lemma then follows when $\ep $ tends  to $0$.
\end{proof}

\begin{lemma} \label{lemmanewton:2}
When $N\ge 2$, $a_1, \dots a_N\in \RR$, $\sigma\in \CC$,  $\opn{Re} {\sigma} >0$, define
$$
F(a_1, \dots , a_{N}; {\sigma}) = \int\limits _0 ^\infty (\varphi _{a_1}
\ast \cdots \ast \varphi _{a_{N-1}})(t) \cos (ta_{N}) \ee^{-{\sigma}t}\,
\dd t.
$$
Then
\begin{equation}\label{newton:3}
F(a_1, \dots ,a_{N};{\sigma}) =
 \frac 1 2 \left(\prod\limits _{1\leq j \leq {N-1}} \frac
 1{a_j^2 -(a_N-\ii\sigma)^2}
+  \prod\limits _{1\leq j \leq {N-1}} \frac
1{a_j^2-(a_N + \ii \sigma)^2}\right). 
\end{equation}
\end{lemma}

\begin{proof}
Since both sides of \eqref{newton:3} depend continuously  in 
$a_1, \dots ,a_N\in \RR$ it is no restriction to assume 
that $a_j^2 \ne a_k^2$ when $j\ne k$.

First notice that when $a$, $b\in \RR$ and $\sigma \in \CC$, $\opn{Re} \sigma >0$, one has
\begin{equation}\label{newton:4}
\int\limits _0 ^\infty 
\varphi _a(t) \cos (tb) \ee^{-{\sigma}t} \, \dd t
=\frac{a^2-b^2+\sigma ^2}{(a^2-b^2+\sigma^2)^2+ 4 b^2\sigma ^2}.
\end{equation}
When $N=2$ \eqref{newton:3} follows directly from this formula.

Assume $N\ge 3$. 
The previous lemma and \eqref{newton:4} give
$$
\begin{gathered}
F(a_1, \dots , a_{N}; {\sigma}) =
\sum\limits_{j=1}^{N-1} \prod_{k\ne j}\frac 1 {a_k ^2-a_j^2}
\int\limits_0^\infty \varphi_{a_j}(t) \cos(ta_N) \ee^{-\sigma t}\dd t\\ 
=\sum _{j=1} ^{N_1} \Big ( \prod _{k\neq j}
  \frac 1 {a_k ^2-a_j^2}\Big ) \frac {a_j^2-a_{N}^2 +
  \sigma ^2}{(a_j^2-a_{N}^2 +
  \sigma ^2)^2+ 4 a_{N}^2\sigma ^2}.
\end{gathered}
$$
We can simplify this  expression  by writing
$$
t_j=a_j^2-a_{N}^2 +\sigma ^2 , \quad 0 \leq j \leq N-1, 
\qquad
\text{and}  \quad
b = 2a_{N}\sigma.
$$
Then
$$
\begin{gathered}
F(a_1, \dots , a_{N}; {\sigma}) = \sum \limits_{j=1}^{N-1} 
\Big (\prod \limits_{k \neq j}
\frac 1{t_k-t_j} \Big ) \frac {t_j}{t_j^2 + b^2}\\
=\frac 1 2 \sum _{j=1}^{N-1}
\Big (\prod _{k \neq j}\frac 1{t_k-t_j} \Big ) 
\frac {1}{t_j-\ii b}+\frac 1 2 \sum _{j=1}^N \Big (\prod _{k \neq
  j, k\leq N}\frac 1{t_k-t_j} \Big ) 
  \frac {1}{t_j+\ii b}\\
= \frac 1 2 \prod _{1\leq j \leq N} \frac 1{t_j-\ii b}+
\frac 1 2 \prod _{1\leq j \leq N} \frac 1{t_j+\ii b}.
\end{gathered}
$$
This finishes the proof of the lemma, after noticing that
$t_j\pm \ii b = a_j^2 -(a_N \mp \ii \sigma)^2$.
\end{proof}

The next lemma is a direct consequence of  Theorem 1.4.2 in \cite{H:I}.

\begin{lemma}\label{chiN:lemma}
There is a sequence $(\chi _N)_1^\infty$  in $C_0^\infty (\RR)$   such
that $\chi _N(t) =1$ when $|t| \leq 1$,  $\chi _N(t)=0$ when $|t|>2$
and
$$
|\chi _N \upp k (t)| \leq   C^k N^k, \quad 0 \leq k \leq 2N+2.
$$
Here $C >0$ is independent of $N$.
 \end{lemma}

In what follows $R$ is an arbitrary positive number.
We set $\chi_{N, R}(t) =\chi_{N}(t/R)$, so that $\chi_{N, 1}
=\chi_N$.
We define
\begin{equation}\label{8}
\begin{gathered}
E_{N,R} =(-1)^{N-1} \int \limits_0 ^\infty Q_{N-1}(x_1, \dots , x_{N-1};t)
\dot k_0(x_N;t) \chi _{N,R}(t)  \dd t , \quad N=2, 3 \dots.
\end{gathered}
\end{equation}

We notice that
\begin{equation}\label{9}
|x_1| + \cdots +|x_{N-1}| = |x_N| \leq 2R \quad \text{in $
  \opn{supp}(E_{N,R})$}
\end{equation} 
and
\begin{equation}\label{10}
  E_{N,R} (x_1, \dots , x_N)= E_N(x_1, \dots , x_N)  \quad \text{when
    $|x_N|\leq R$}.
\end{equation}
We shall derive estimates for the Fourier transform 
$\Cal F E_{N,R}(\xi _1, \dots ,\xi _N)$ of $E_{N,R}$.
We notice here that, due to the homogeneity of $Q_{N-1}(\cdot; t)$ and of  
$\dot k_0(\cdot;t)$ and to the definition of $\chi_{N, R}$, we have
$$
 E_{N, R}(Rx_1, \dots, R x_N)= R^{2N-2} R^{-Nn} E_{N, 1}(x_1, \dots x_N).
$$
It follows that
\begin{equation}\label{homog}
 (\Cal F E_{N,R})(\xi _1, \dots , \xi _N) = R^{2N-2}  {\Cal F} E_{N, 1}(R\xi_1, \dots, R\xi_N).
\end{equation}
Therefore it is enough to establish estimates 
for ${\Cal F} E_{N, 1}$.

The distribution  $E_{N, 1}(x_1,. \dots , x_N)$ is rotation invariant in the
variables $x_1, \dots , x_N$ and compactly supported.
The Fourier transform $\Cal F E_{N,1}(\xi _1, \dots
,\xi _N)$ of $E_{N,1}$ is 
smooth and rotation invariant in each variable $\xi _j$.
We
define $F_{N}(r_1, \dots , r_N)$ when $r _j \geq 0$ by
\begin{equation}\label{newton:1001}
(\Cal F E_{N,1})(\xi _1, \dots , \xi _N) = F_{N}(r_1, \dots ,r_N)
\quad \text{when $r_j=|\xi _j|$.}
\end{equation}
Hence we need estimates of $F_N$.

Consider  $\gamma> 0$. 
Let us define the  functions  $h_\gamma(r,s)$ through
$$
h_\gamma(r,s) = (\gamma+|r-s|)^{-1}(\gamma+|r+s|)^{-1}.
$$

\begin{lemma}\label{ineq:lemma}
When $s$, $t\in \RR$, one has
$$ 
1+|s-t|\ge \frac {1+|s|}{1+|t|}.
$$
Consequently
$$
h_\gamma(s, r+t) \le \gamma^{-2} (\gamma+|t|)^2 h_\gamma(s, r)
$$ 
when $s$, $t$, $r\in \RR$.
\end{lemma}

\begin{proof}
The lemma follows from the inequalities
$$
\begin{gathered}
1+|s-t|  \geq 
1 + \frac{|s-t|}{1+|t|}\ge 1 + \frac{|s|-|t|}{1+|t|}
=\frac{1+|s|}{1+|t|}.
\end{gathered}
$$
\end{proof}

The estimate of $F_{N}$ that we need is contained in the next lemma.
\begin{lemma}\label{lemmanewton:3}
There is a constant $C$,  which does not depend on $N$ and $\gamma$, 
such that 
\begin{equation}\label{11}
|F_{N} (r_1, \dots , r_N)| \leq C^N N^{2N+1} \gamma^{-(2N+1)} \ee^{2\gamma} 
\prod _{1 \leq j \leq N-1}h_\gamma (r_j,r_N).
\end{equation}
\end{lemma}

\begin{proof}
 It
follows from \eqref{8} and \eqref{en:2}  that
\begin{equation}\label{13}
F_{N}(r_1, \dots ,r_N) = (-1)^{N-1}\int\limits _{-\infty
} ^\infty {\Phi}_N(r_1, \dots , r_N,t)\chi _{N}(t) \, \dd t
\end{equation}
where
$$
{\Phi}_N(r_1, \dots , r_N,t)= (\varphi _{r_1} \ast \cdots \ast \varphi
_{r_{N-1}})(t)\cos (tr_N).
$$
As a function of $t$, ${\Phi}_N(r_1, \dots ,r_N,t)$   is supported 
in $[0,\infty)$ and  of polynomial growth
at infinity. 
 
Define
$$
\widetilde{{\Phi}}_{N, \gamma}(r_1, \dots , r_N,t)= 
\ee^{-\gamma t}{\Phi}_N(r_1, \dots , r_N,t), \quad
\widetilde{\chi} _{N, \gamma}(t) = \ee^{\gamma t}{\chi} _{N}(t).
$$
Then
\begin{equation}\label{14}
\begin{gathered}
F_{N}(r_1, \dots , r_N)= 
\int\limits _{\RR} \widetilde {\Phi} _{N, \gamma}(r_1, \dots ,r_N,t) 
\widetilde \chi _{N,\gamma}(t) \, \dd t \\
=(2\pi )^{-1}\int\limits _{\RR}
(\Cal F \widetilde {\Phi} _{N, \gamma})(r_1, \dots ,
r_N,{\tau}) (\Cal F \widetilde \chi _{N,\gamma})(-{\tau})\, \dd {\tau},
\end{gathered}
\end{equation}
where the Fourier transform is taken  in the variable $t$. 
We notice that
$$
(\Cal F \widetilde {\Phi}_{N, \gamma} )(r_1, \dots ,r_N, \tau) = 
\int  {\Phi}_N(r_1,
\dots r_N, t) \ee^{-{\sigma}t}\dd t
=F(r_1, \dots ,r_N; \sigma),
\quad {\sigma} = \gamma+\ii{\tau}.
$$
Then  an application of Lemma~\ref{lemmanewton:2} gives  the estimate
\allowdisplaybreaks
\begin{gather}
|(\Cal F \widetilde {\Phi}_{N, \gamma})(r_1, \dots , r_N,{\tau})| 
\label{15}
\\
\leq \frac 1 2 \prod _{1\leq j \leq N-1}|r_j^2
-(r_N+\ii{\sigma})^2|^{-1}
+\frac 1 2 \prod _{1\leq j \leq N-1}|r_j^2
-(r_N-\ii{\sigma})^2|^{-1}
\nonumber
\\
=\frac  1 2 \prod _{1\leq j \leq N-1} |r_j-(r_N-\tau)-\ii\gamma|^{-1}
|r_j+(r_N-\tau)+\ii\gamma|^{-1}
\nonumber
\\
+\frac  1 2 \prod _{1\leq j \leq N-1} |r_j-(r_N+\tau)-\ii\gamma|^{-1}
|r_j+(r_N+\tau)+\ii\gamma|^{-1}
\nonumber
\\
\leq 2^{N-2}\prod _{1\leq j \leq
  N-1}(\gamma+|r_j-(r_N-\tau)|)^{-1}(\gamma+|r_j+(r_N-\tau)|)^{-1} 
\nonumber
\\
+ 2^{N-2}\prod _{1\leq j \leq
  N-1}(\gamma+|r_j-(r_N+\tau)|)^{-1}(\gamma+|r_j+(r_N+\tau)|)^{-1}
\nonumber
\\
= 2^{N-2}\prod _{1\leq j \leq N-1}h_\gamma(r_j,r_N-{\tau})+ 2^{N-2}\prod _{1\leq j \leq N-1}h_\gamma(r_j,r_N+{\tau}).
\nonumber
\end{gather}

Next we see that
$$
\begin{gathered}
\Cal F \widetilde \chi _{N, \gamma} (-{\tau}) =\int \chi _{N, \gamma}(t)
\ee^{t(\gamma+\ii{\tau})}\, \dd t\\
=(\gamma+\ii{\tau}) ^{-(2N+2) }\int \chi _{N,\gamma} \upp {2N+2}(t)
\ee^{t(\gamma+\ii{\tau})}\, \dd t.
\end{gathered}
$$ 
From this and Lemma~\ref{chiN:lemma} we deduce that there is a constant  $C$,  which is
independent of $N$ and $\gamma$, such that
$$
|\Cal F \widetilde \chi _{N,\gamma}(-{\tau})| \leq
C^N N^{2N+2} \ee^{2\gamma}(\gamma +|\tau|)^{-2N-2}.
$$

Then \eqref{14}, \eqref{15} and the above inequality, together with Lemma~\ref{ineq:lemma}, give
$$
\begin{gathered}
|F_{N} (r_1, \dots , r_N)| \leq 
C^N N^{2N+2} \ee^{2\gamma}
\int\limits _{-\infty}^\infty
(\gamma +|\tau|)^{-2N-2}
\Big (\prod _{1 \leq j \leq N-1}h_\gamma(r_j,r_N-\tau)\Big
) \dd{\tau}\\
\le
C^N N^{2N+2} \gamma^{-2(N-1)}  \ee^{2\gamma}
(\int\limits _{-\infty}^\infty
(\gamma +|\tau|)^{-4} \dd\tau)
\Big (\prod _{1 \leq j \leq N-1}h_\gamma(r_j,r_N)\Big)
\\
\le 
C^N \gamma^{-(2N+1)} N^{2N+2} \ee^{2\gamma}
\prod _{1 \leq j \leq N-1}h_\gamma(r_j,r_N) 
\\
\le 
(2C)^N \gamma^{-(2N+1)} N^{2N+1} \ee^{2\gamma}
\prod _{1 \leq j \leq N-1}h_\gamma(r_j,r_N).
\end{gathered}
$$
This finishes the proof.
\end{proof}

The following theorem gives the estimate we need for the Fourier transform of $E_{N, R}$.

\begin{theorem}\label{newtonthm}
There is a constant  $C>0$, which depends on $n$ only, such that  
$$ 
| ({\Cal F}E_{N, R})(\xi_1,\dots  ,\xi_N)|\le 
C^N  (N/(R\gamma))^{2N+1} e^{2 R \gamma} \prod\limits_{1\le j \le N-1} 
h_\gamma(|\xi_j|, |\xi_N|),\qquad  \xi_1, \dots \xi_N\in \RR^n $$ 
for every $N\ge 2$, $R>0$ and $\gamma>0$.
\end{theorem}

\begin{proof}
Let $R>0 $.
The identity \eqref{homog} and  previous lemma show 
that there is a constant $C>0$, which depends on $n$ only, such that
$$ 
| ({\Cal F}E_{N, R})(\xi_1,\dots  ,\xi_N)|\le 
C^N (N/\gamma)^{2N+1}R^{2N-2} e^{2\gamma} \prod\limits_{1\le j \le N-1} 
h_\gamma(R|\xi_j|, R|\xi_N|),$$ 
when $\xi_1, \dots ,\xi_N\in \RR^n $, 
for every $N\ge 2$, $R>0$ and  $\gamma>0$.
This in turn shows that, with the same $C$, one has
$$
| ({\Cal F}E_{N, R})(\xi_1,\dots  ,\xi_N)|\le 
C^N  (N/\gamma)^{2N+1} 
e^{2\gamma} \prod\limits_{1\le j \le N-1} 
h_{\gamma/R}(|\xi_j|, |\xi_N|).
$$ 
The theorem follows by replacing $\gamma/R$ by $\gamma$.
\end{proof}

\section{$L^2$-Sobolev estimates for $B_N$}

We introduce an $N$-linear version of $B_N$, $N\ge 2$.
Namely,  for $\vec v=(v_1, \dots, v_N)$, $v_j \in C_0^\infty(\RR^n)$, define
\begin{equation}\label{bn:0001}
\begin{gathered}
(\BB_N\vec v)(x) =  
\int\limits_{(\RR^n)^{N}}
E_{N}(y_1, \dots, y_{N})\\
 v_1(x-\frac{y_N}{2} -Y_0)
v_2(x-\frac{y_N}{2} -Y_1)
\cdots v_N(x-\frac{y_N}{2} -Y_{N-1}) \dd \vec{y}.
\end{gathered}
\end{equation}
Here the $Y_k$:s are defined as in Proposition~\ref{prop:en1}, that is, 
$$ 
Y_0 = \frac{1}{2} \sum\limits_{1}^{N-1} y_j, \qquad
Y_k = Y_0 -\sum\limits_{j=1}^k y_j, \quad k=1, \dots, N-1.
$$ 
Then $\BB_N\vec v$ is a smooth compactly supported function in $\RR^n$
and  $B_Nv = {\BB}_{N}(v, \dots, v)$ for every $v\in C_0^\infty(\RR^n)$.
Therefore the result in Theorem~\ref{mainthm:1}  is contained in the next theorem.
Here and in the rest of the section we use the notation 
$$
m=\frac{n-3}{2}.
$$

\begin{theorem}\label{mainthm:2}
Assume that $0<\ep<1$,  $s_j \ge m$ and 
$a_j = \min (s_j-m, 1-\ep)$, $j=1, \dots ,N$. Set
$$ \sigma= \min(s_j-a_j) +\sum\limits_{j=1}^N a_j.$$
Then there is a constant $C$ which is independent of the $s_j$, 
but may depend 
on $\ep$ and $n$,  such that 
\begin{equation}\label{main:1}
\normrum{\BB_N\vec v}{H_{(\sigma )}(B(0, R))}^2 \leq C^N  
N^{2\min(s_j-a_j-m) }(R/N)^{N-1} \prod_1^N \normrum{v_j}{(s_j)}^2, 
\end{equation}
for every $N\ge 2$, $R>0$, $v_1, \dots ,v_N\in C_0^\infty(B(0, R))$.
\end{theorem}

The present section is devoted to the proof of the above result. 
We start with some preparations.

Let $R >0$ and recall that the distributions $E_{N,R}\in \Cal
E'((\RRn)^N) $ were defined in \eqref{8}. 
When $\vec v=(v_1, \dots , v_N)$, $v_j\in C_0^\infty(\RR^n)$, we consider 
\begin{equation}\label{def:betaN}
\begin{gathered}
(\BB _{N,R}\vec v)(x) = 
\int\limits_{(\RR^n)^{N}}E_{N,R}(y_1, \dots, y_{N})
\\
 v_1(x-\frac{y_N}{2} -Y_0)
v_2(x-\frac{y_N}{2} -Y_1)
\cdots v_N(x-\frac{y_N}{2} -Y_{N-1}) \dd \vec{y}. 
\end{gathered}
\end{equation}
It is easy to see that $\BB _{N,R}\vec v$ is a smooth 
compactly supported function in $\RRn$.
The following
lemma  gives the connection between $\BB _{N,R}\vec v
 $ and $\BB _{N}\vec v$.

\begin{lemma}\label{lemma:5.2}
Assume $v_1, \dots , v_N \in C_0^\infty (B(0,R))$. Then
$(\BB _{N,4R} \vec v)(x) =(\BB _N\vec v)(x)$ in a neighbourhood of
$\overline{B(0,R)}$  and $\BB _{N,2(N-1)R}\vec v = \BB _N \vec v$.
\end{lemma}

\begin{proof}
Choose $\ep >0$ such that the $v_j$ are supported in $B(0,R-\ep)$ and
define
\begin{equation}\label{vx:def}
V_x(\vec y) = v_1(x-y_N/2-Y_0)\cdots  v_N(x-y_N/2-Y_{N-1}).
\end{equation}
 Since $Y_0+Y_{N-1}=0$, it
follows that
$$
|2x-y_N| = |(x-y_N/2-Y_0)+(x-y_N/2-Y_{N-1})| \leq 2R-2\ep
$$
when $\vec y \in \opn{supp} (V_x)$. 
When $|x|<R+\ep/2$ we see that $|y_N|<4R$ when
$\vec y$ is in the support of $V_x$ and, since $E_{N,4R}=E_{N}$ 
when $|y_N|<4R$, it follows
that $(\BB _{N, 4R}\vec v)(x)= (\BB _N\vec v)(x)$ when $|x|<R+\ep /2$. 
This
proves the first assertion. 
When proving the second assertion we
notice that 
$$
|y_j| = |(x-y_N/2-Y_{j-1}) -(x-y_N/2-Y_j)| < 2R
$$
when $1 \leq j \leq N-1$ and $\vec y \in  \opn{supp} (V_x)$, hence
$
\sum _{1}^{N-1}|y_j|< 2(N-1)R$. 
This shows that the support of
$V_x$ does not intersect the support of $E_{N,2(N-1)R} -E_N$, hence
$\BB _{N,2(N-1)R}\vec v = \BB _N \vec v$. 
\end{proof}

Let $ \vec s = (s_1, \dots , s_N)$ be a sequence of nonnegative 
real numbers and let $\sigma \in \RR$, $ N\geq 2$, $ R > 0$. Define 
\begin{equation}\label{anr:def}
\begin{gathered}
A(N,R,\vec s, {\sigma}) = \\ \sup _{{\xi}_N}\idotsint
(1+4|{\xi}_N|^{2})^{\sigma}  
 |(\Cal F E_{N,R})({\xi}_1, \dots ,{\xi}_N)|^2M_{\vec s}(\xi_1 , \dots
, {\xi} _N)^2 \dd{\xi}_1 \dots \dd{\xi}_{N-1},
\end{gathered}
\end{equation}
where 
$$
M_{\vec s}(\xi _1, \dots , {\xi}_N) = \hake{\xi _1+ \xi _N}^{-s_1}\hake{\xi _2
  - \xi _1}^{-s_2} \cdots \hake{\xi _N - \xi _{N-1}}^{-s_N}.
$$
Then $0 \leq A(N,R,\vec s, {\sigma}) \leq \infty$.

\begin{lemma}\label{lemmamainest}
We have that
\begin{equation}\label{bn:mainest}
\begin{gathered}
\normrum{\BB_{N,R}\vec v}{(\sigma)}^2 \leq (2{\pi})^{n(1-N) }
A_{N,R}(s_1, \dots , s_N,{\sigma}) \prod _1 ^N \normrum{v_j}{(s_j)}^2
\end{gathered}
\end{equation}
for every  $v _j \in C_0^\infty (\RRn)$, $1\le j\le N$.
\end{lemma}

\begin{proof}
Let $V_x$ be defined as in \eqref{vx:def}. 
In order to compute the
Fourier transform of $V_x$ we introduce the linear map $L$ in
$(\RRn)^N$  through  $L \vec z = \vec y$, where
$$
y_j = z_j-z_{j+1}, \ 1 \leq j \leq N-1,\  y_N = z_1+z_N.
$$
It is easily seen that $\det (L)=2^n$ and that $y_N/2+ Y_{j-1} =z_j$
when $1\leq j \le N$.
Therefore we may write
$$ V_x(\vec y) = (v_1\otimes \cdots \otimes v_N) 
(- L^{-1} (y_1, \dots y_{N-1}, y_N-2x)). 
$$
Hence
$$ {\Cal F}V_x(-\xi_1, \dots, -\xi_N)= 2^n \ee^{2\ii\scalar{x}{\xi_N}}
({\hat v}_1 \otimes \cdots \otimes{\hat v}_N)( \,  L'\vec \xi).
$$
Here $L'$ denotes the transpose of $L$. 
It is easy to see that $L' \vec \xi =\vec \eta$, where
$$
{\eta}_1 ={\xi}_1+{\xi}_N, \ {\eta}_j = {\xi}_j-{\xi}_{j-1}, \quad
2\leq j \leq N.
$$
It follows that
$$
\begin{gathered}
(\Cal F V_x)(- \xi _1, \dots , - \xi _N) = 
2^n \ee^{2\ii\scalar x{\xi _N}} \widehat v _1( \xi _1+ \xi _N) \widehat v
_2( \xi _2- \xi _1) \cdots \widehat v _N( \xi _N - \xi _{N-1}).
\end{gathered}
$$
Write $w_j = \Cal F \hake{D}^{s_j} v_j$   and 
$$
W(\vec \xi ) = w_1(\xi _1+\xi_N) w_2( \xi _2- \xi _1) \cdots  w_N( \xi
_N- \xi _{N-1}).
$$
It follows from \eqref{def:betaN} and the  computations above that
$$
\begin{gathered}
(\BB_{N,R}\vec v)(x) =
(2{\pi})^{-nN}
\int (\Cal F E_{N,R})(\vec \xi) (\Cal F V_x)(- \vec \xi)  \dd\vec \xi
\\=
(2{\pi} )^{-nN}2^n \int  \ee^{2\ii \scalar x{\xi _N}}  {\beta}({\xi}_N) 
    \dd{\xi}_N\\
=(2{\pi})^{-nN}\int  \ee^{\ii\scalar x {{\xi}_N}} {\beta}({\xi}_N/2) \dd{\xi}_N,
\end{gathered}
$$ 
where
$$
\begin{gathered}
{\beta}({\xi}_N)= 
\idotsint (\Cal F E_{N,R})(\vec {\xi}) \widehat{v}_1( \xi _1 +\xi _N) \widehat v _2(\xi _2-\xi _1)
\cdots \widehat v _N(\xi _N - \xi _{N-1}) \dd{\xi} _1 \cdots  
\dd{\xi}_{N-1}\\
=\idotsint (\Cal F E_{N,R})(\vec {\xi}) M_{\vec s}(\vec {\xi}) 
W(\vec{\xi}) \dd{\xi}_1\cdots \dd{\xi}_{N-1}.
\end{gathered}
$$
This shows that
\begin{equation}\label{est1}
\begin{gathered}
\normrum{\BB_{N,R}\vec v}{(\sigma) }^2 = (2{\pi})^{-2n(N-1/2)}\int
\hake{\xi _N}^{2{\sigma}} |{\beta}({\xi}_N/2)|^2 \dd{\xi}_N \\
= 2^n (2{\pi})^{-2n(N-1/2)} \int (1+|4{\xi}_N|^2)^{\sigma}|{\beta}({\xi}_N)|^2 \,
\dd{\xi}_N\\
\leq 2^n (2{\pi} )^{-2n(N-1/2)}\int  \Big \{ (1+4|{\xi}_n|^2)^{{\sigma}}\Big ( \idotsint |(\Cal F
E_{N,R})(\vec {\xi}) |^2 M_{\vec s}({\xi})^2  \dd{\xi}_1 \cdots
\dd{\xi}_{N-1} \Big )
\\
\Big (\idotsint  |W(\vec {\xi})|^2 \dd{\xi} _1\cdots d{\xi}_{N-1}\Big )\Big \}
\dd{\xi}_N.
\\
\leq 2^n (2{\pi})^{-2n(N-1/2)}A(N,R,\vec s,{\sigma}) \int |W(\vec
{\xi})|^2  \dd\vec {\xi}. 
\end{gathered}
\end{equation}
The proof is then completed by the observation that
$$
W(\vec {\xi})= (w_1 \otimes w_2 \otimes \cdots \otimes w_N)(L'\vec {\xi}).
$$
 It  follows that
$$
\begin{gathered} 
\int |W(\vec {\xi})|^2 \dd\vec {\xi} = 2^{-n}\int |(w_1\otimes \cdots
\otimes w_N)(\vec {\xi})|^2 \dd\vec {\xi}\\
= 2^{-n} \prod _1 ^N \norm{w_j}^2 = 2^{-n} (2{\pi})^{nN} \prod _1 ^N
\norm{\hake{D}^{s_j} v_j}^2 
=2^{-n}(2{\pi})^{nN}\prod _1 ^N
\normrum{v_j}{(s_j)}^2.
\end{gathered}
  $$ 
The lemma follows if this is inserted into \eqref{est1}.
\end{proof}

We shall arrive at estimates for  $B_{N,R}\vec v$ by combining  the
inequality \eqref{bn:mainest}  with estimates for the expression
$A(N,R, \vec s, \sigma)$ in \eqref{anr:def}.  
The  following lemma
will be needed.
\begin{lemma}\label{5.4}
Assume $0<\ep <1$. Then there is a constant $C=C_{n,\ep}$  such that
\begin{equation}\label{hgamma:est}
\int h_{\gamma} ^2(|{\xi}|,{\rho})\hake {{\xi}-{\eta}}^{-2s} \dd{\xi}
\leq C {\gamma}^{-1}\hake{\rho}^{2m-2s}, 
\end{equation}
when ${\eta} \in
  \RRn$, ${\rho} \geq 0$, ${\gamma} > 0$, $m \leq s \leq m+1-\ep$. 
\end{lemma}

\begin{proof}
 Assume $r>0$ and ${\eta} \in \RRn
\setminus 0$. Set
$$
f_s(r,{\eta}) = \int\limits _{\bS^{n-1}} \hake {r{\theta}-{\eta}}^{-2s}\,
\dd{\theta}.
$$
If $ u= \scalar {\theta} {\eta}/|{\eta}|$ then  a simple computation
shows that
$$
\hake{ r{\theta}-{\eta}}^2  \geq 1+ r^2(1-|u|).
$$
If $f$ is  a continuous function, and $c_{n-2}$ is the area of the
$n-2$-dimensional unit sphere,  then 
$$
\int\limits _{\bS^{n-1}} f(\scalar {\theta} {\eta}/|{\eta}|) \dd{\eta} = 
c_{n-2} \int _{-1}^1 f(t) (1-t^2)^m\dd t.
$$
This shows that
$$
\begin{gathered}
f_s(r,{\eta})\leq c_{n-2} \int\limits _{-1}^1(1+r^2(1-|t|))^{-s}(1-t^2)^m \dd t 
\\
\leq 2^{m+1}c_{n-2} \int\limits _{0}^1(1+r^2(1-t))^{-s}(1-t)^m \dd t 
\leq 2^{m+1}c_{n-2} \int\limits _0 ^1 (1+r^2t)^{-s}t^m \dd t\\
\leq  2^{m+1}c_{n-2} \hake r ^{-2s}\int\limits _0^1 t^{m-s} \dd t .
\end{gathered}
$$
This gives the estimate
\begin{equation}\label{extra:4001}
f_s(r,{\eta}) \leq C_1 \hake  r ^{-2s},
\end{equation}
where $C_1=  2^{m+1}c_{n-2} /\ep$, for $\eta\in \RR^n\setminus 0$. 
This inequality clearly holds for $\eta=0$ as well. 

Using \eqref{extra:4001} and introducing polar coordinates in the
integration
one gets 
\begin{equation}\label{40011}
\int h_{\gamma} ^2(|{\xi}|,{\rho})\hake {\xi -{\eta}}^{-2s} 
\dd {\xi}
\leq C_1 \int \limits_0^\infty h_{\gamma} ^2( r, {\rho})r^{2m+2}\hake r^{-2s} 
\dd r.
\end{equation}

Assume first that $\rho\ge 1$.
Then
\begin{equation}\label{40012}
\begin{gathered}
\int\limits_0^\infty 
h_{\gamma} ^2( r, {\rho})r^{2m+2}\hake r^{-2s} 
\dd r \\
= \int\limits_0^\infty 
\frac{1}{(\gamma+|r-\rho|)^2}
\,
\frac{1}{(\gamma+r+\rho)^2}
\,
\frac{r^{2m+2}}{(r^2+1)^s}\dd r\\
\le 
\frac{1}{\rho^{2(s-m)}}
\int\limits_0^\infty \frac {1}{(\gamma+|r-\rho|)^2}
\,\frac{r^{2m+2}}{(r^2+1)^{m+1}}\dd r\\
\le \frac{2^{2(s-m)}}{(\rho+1)^{2(s-m)}} 
\int\limits_{\RR} \frac{1}{(\gamma+|r|)^2} \dd r
\le 2^3 \gamma^{-1} \hake{\rho}^{-2(s-m)}.
\end{gathered}
\end{equation}
Assume next that $\rho<1$. 
Then
\begin{equation}\label{40013}
\begin{gathered}
\int\limits_0^\infty 
h_{\gamma} ^2( r, {\rho})r^{2m+2}\hake r^{-2s} 
\dd r \\
\le  \int\limits_0^\infty 
\frac{1}{(\gamma+|r-\rho|)^2}
\,
\frac{r^{2m}}{(r^2+1)^s}\dd r
\le 
\int\limits_{\RR} \frac{1}{(\gamma+|r|)^2}\dd r
\\
\le 2^{2(s-m)+1}\gamma^{-1} {(\rho^2+1)^{-(s-m)}}
\le 2^3 \gamma^{-1} \hake{\rho}^{-2(s-m)}.
\end{gathered}
\end{equation}
Combining \eqref{40011}, \eqref{40012} and \eqref{40013} we see that 
the lemma holds with $C=2^3 C_1$.
\end{proof}

Now we are going to estimate $A(N,R,\vec s, \gamma)$. 
Recall  that Theorem~\ref{newtonthm} gives that
\begin{equation}\label{extra:4002}
\begin{gathered}
|(\Cal F E_{N,R})({\xi}_1, \dots ,{\xi}_N)| \leq\
C^N (N/(\gamma R))^{2N+1}
e^{2R{\gamma}} \prod
_{1\leq j \leq N-1} h_{\gamma}(|{\xi}_j|,|{\xi}_N|),
\end{gathered}
\end{equation}
where ${\gamma} > 0$, $ N\geq 2$,  $R> 0$ and
the constant $C$ is
independent of these parameters.
We notice that
\begin{equation}\label{extra:4003}
\begin{gathered}
M_{\vec s} ({\xi}_1, \dots , {\xi}_N) \leq
2^{s_2} \hake {\xi _1+ \xi _N}^{-s_1} M_{s_2, \dots ,s_N}(\xi _2,
\dots ,\xi _N) \\+ 2^{s_1} \hake {\xi _2-\xi _1}^{-s_2} M_{s_1, s_3,
  \dots ,s_N} (\xi _2, \dots , {\xi} _N).
\end{gathered}
\end{equation}
In fact, since 
$$
|{\xi}_2+{\xi}_N| \leq |{\xi}_2-{\xi}_1|+ |{\xi}_1+{\xi}_N|,
$$
either $|{\xi}_2-{\xi}_1|  \geq |{\xi}_2+{\xi}_N|/2$ or
$|{\xi}_1+{\xi}_N|\geq |{\xi}_2+{\xi}_N|/2$. 
In the first case 
$$
M_{\vec s}({\xi}_1, \dots ,{\xi}_N)\leq 2^{s_2} \hake{{\xi}_1+{\xi}_N}
^{-s_1} M_{ s_2, \dots ,s_N}({\xi}_2, \dots , {\xi}_N)$$
and in the second case
$$
M_{\vec s}({\xi}_1, \dots ,{\xi}_N)\leq 2^{s_1} \hake{{\xi}_1-{\xi}_2}
^{-s_2} M_{s_1, s_3, \dots ,s_N}({\xi}_2, \dots , {\xi}_N).$$

When $N \geq 2$ we define
$$
\begin{gathered}
T_{N,{\gamma}}({\xi}, \vec s) = \idotsint \Big (\prod _{1\leq j \leq
  N-1} h_{\gamma}(|{\xi}_j|, |{\xi}|) \Big )^2 M_{\vec s}^2({\xi}_1,
\dots ,{\xi}_{N-1}, {\xi}) \dd{\xi}_1 \dd{\xi}_2\dots \dd{\xi}_{N-1}. 
\end{gathered}
$$

Let $0 <\ep <1$ and assume that $m \leq s_j \leq m+1-\ep$ when $0
\leq j \leq N$.
It follows from  \eqref{extra:4002} and \eqref{anr:def} that
\begin{equation}\label{anrest:2}
A(N,R,\vec s , {\sigma}) \leq  C^N (N/(\gamma R))^{4N+2} e^{4R{\gamma}}
\sup _{\xi}( \hake{2\xi}^{2\sigma} T_{N,{\gamma}}({\xi};\vec s)).
\end{equation}  
Here, and in what follows, $C$ denotes constants that are independent
of $N$, $R$, $\vec s$, ${\sigma}$, ${\gamma}$ (but may depend on $\ep$ and dimension $n$).

Assume $N \geq 3$. 
From \eqref{extra:4003} follows that
$$
\begin{gathered}
T_{N,{\gamma}}({\xi};\vec s) \leq  2^n \Big ( \int h_{\gamma}
^2(|{\xi}_1|,|{\xi}|) \hake{{\xi}_1+{\xi}}^{-2s_1} \dd{\xi}_1) \\
\idotsint \Big (\prod _{2\leq j \leq N-1}
h_{\gamma}^2(|{\xi}_j|,|{\xi}|)\Big )M_{s_2, \dots ,s_N}^2({\xi}_2,
\dots, {\xi}_{N-1},  {\xi})\dd{\xi}_2\cdots \dd{\xi}_{N-1}
\\
+ 
 2^n \Big ( \int h_{\gamma}
^2(|{\xi}_1|,|{\xi}|) \hake{{\xi}_1-{\xi}_2}^{-2s_2} \dd{\xi}_1) \\
\idotsint \Big (\prod _{2\leq j \leq N-1}
h_{\gamma}^2(|{\xi}_j|,|{\xi}|)\Big )M_{s_1,s_3, \dots ,s_N}^2({\xi}_2,
\dots, {\xi}_{N-1},  {\xi})\dd{\xi}_2\cdots \dd{\xi}_{N-1}.
\end{gathered}
$$
From Lemma~\ref{5.4} we get the 
estimate
\begin{equation}\label{pn:12}
\begin{gathered}
T_{N,\gamma}({\xi};\vec s) \leq (C/2) {\gamma} ^{-1} \hake {\xi}
^{2m-2s_1} T_{N-1,{\gamma}} ({\xi};s_2, \dots , s_N)\\
+  (C/2) {\gamma} ^{-1} \hake {\xi}
^{2m-2s_2} T_{N-1,{\gamma}} ({\xi};s_1, s_3,\dots , s_N).
\end{gathered}
\end{equation}
Another application  of  Lemma~\ref{5.4} gives 
$$
\begin{gathered}
T_{2,{\gamma}}({\xi};s_1, s_2) 
= \int h_{\gamma}^2(|{\xi}_1|,|{\xi}|) \hake{{\xi}_1+{\xi}}^{-2s_1}
\hake{{\xi}_1-{\xi}}^{-2s_2} \dd{\xi} _1 
\\
\leq \hake{\xi} ^{-2s_1} \int h_{\gamma} ^2(|{\xi}_1|, |{\xi}|)
\hake{{\xi}_1-{\xi}}^{-2s_2}\dd{\xi}_1
+ \hake  \xi ^{-2s_2} \int h_{\gamma} ^2(|{\xi}_1|, |{\xi}|)
\hake{{\xi}_1+{\xi}}^{-2s_1}\dd{\xi}_1 
\\
\leq C{\gamma}^{-1}\hake {\xi} ^{2m-2s_1-2s_2}  
\end{gathered}
$$
where we may assume that $C$ is the same  constant as in
\eqref{pn:12}.
From this we deduce  that the inequality
\begin{equation}\label{pnest:1}
T_{N,{\gamma}}({\xi}; \vec s) \leq 
C ^{N-1}{\gamma} ^{-(N-1)}\hake {\xi}^{2((N-1)m-s_1- \cdots -s_N)}
\end{equation}
holds when  $N=2$. 
Applying  \eqref{pn:12} together with an induction
argument we obtain that \eqref{pnest:1} holds for every $N \geq 2$. 
Since $m\le s_j < m+1$ we get (with another $C$)
\begin{equation}\label{pnest:11}
T_{N,{\gamma}}({\xi}; \vec s) \leq 
C ^{N-1}{\gamma} ^{-(N-1)}\hake {2\xi}^{2((N-1)m-s_1- \cdots -s_N)}.
\end{equation}

Assume now that $s_j \ge m$, $j=1, \dots ,N$, but not necessarily
$s_j < m+1$, and let $0< \ep< 1$.
Set $a_j = \min(s_j-m, 1-\ep)$, $j=1, \dots, N$.
We notice that 
$$ |\xi_1+\xi_N|+ |\xi_2-\xi_1|+\cdots +|\xi_N -\xi_{N-1}|\ge 2 |\xi_N|,
$$
and therefore
$$ 
\max(|\xi_1+\xi_N|, |\xi_2-\xi_1|, \cdots |\xi_N -\xi_{N-1}|) \ge 
2|\xi_N|/N.
$$
It follows that
$$ 
\hake{\xi_1+\xi_N}^{-1}\hake{\xi_2-\xi_1}^{-1}\cdots \hake{\xi_N -\xi_{N-1}}^{-1}
\le (1+ 4|\xi_N|^2/N^2)^{-1/2} \le N \hake{2\xi_N}^{-1}.
$$ 
Then we may write
$$ 
M_{\vec s}(\vec \xi) \le N^{\min(s_j-(a_j+m))} \hake{2\xi_N}^{- \min(s_j-(a_j+m))}
M_{(m+a_1, \dots, m+a_N)}(\vec \xi).
$$
This implies that
$$ 
T_{N, \gamma}(\xi; \vec s) \le N^{2\min(s_j-(a_j+m))} \hake{2\xi}^{- 2\min(s_j-(a_j+m))}
T_{N, \gamma}(\xi; m+a_1, \dots, m+a_N).
$$
Then \eqref{pnest:11} gives
\begin{equation}\label{pnest:2}
T_{N, \gamma}(\xi, \vec s) \le   C^{N-1} N^{2\min(s_j-(a_j+m))} 
{\gamma}^{-(N-1)}\hake
{2\xi}^{-2\min(s_j-a_j) -2\sum\limits_1^N a_j}.
\end{equation}

Combining \eqref{anrest:2} with \eqref{pnest:2}  we get the following
lemma.
\begin{lemma}\label{5.5}
Assume that $0<\ep<1$,   $s_j \ge m$  and $a_j = \min(s_j-m, 1-\ep)$, 
$j=1, \dots ,N$.
Set
$$ \sigma= \min(s_j-a_j) +\sum\limits_{j=1}^N a_j$$
Then there is a constant $C$ which is independent of the $s_j$, but may depend 
on $\ep$ and $n$, such that 
\begin{equation}\label{anrs:est}
A(N,R,\vec s,{\sigma}) \leq C^N N^{2\min(s_j-a_j-m)} (N/(\gamma R))^{4N+2}
{\gamma}^{-(N-1)}e^{4R{\gamma}}
\end{equation}
for every $\gamma>0$, $R>0$ and $N\ge 2$.
\end{lemma}

Next we recall \eqref{bn:mainest}  which, together with the previous 
lemma, gives the next proposition.

\begin{proposition}\label{bngamma}
Assume that $0<\ep<1$,   $s_j \ge m$  and $a_j = \min(s_j-m, 1-\ep)$, $j=1, \dots ,N$.
Set
$$ \sigma= \min(s_j-a_j) +\sum\limits_{j=1}^N a_j.$$
Then there is a constant $C$ which is independent of the $s_j$, 
but may depend 
on $\ep$ and $n$, such that 
\begin{equation}\label{bnest:3}
\normrum{\BB_{N,R}\vec v}{(\sigma )}^2 \leq C^N  N^{2\min(s_j-a_j-m)} (N/(\gamma R))^{4N+2} 
{\gamma}^{-(N-1)} e^{4R{\gamma}}\prod\limits_1^N \normrum{v_j}{(s_j)}^2 , 
\end{equation}
for every $N\ge 2$,  $v_1, \dots ,v_N\in C_0^\infty(\RR^n)$, $R>0$ and 
 $\gamma> 0$.
\end{proposition}
 
 Theorem~\ref{mainthm:2} follows from the previous proposition  and Lemma~\ref{lemma:5.2}, by replacing $R$ by $4R$ and taking $\gamma= N/(4R)$.
 When replacing $R$ by $2(N-1)R$ and taking $\gamma=1/R$, we obtain the 
 following corollary, where we use Lemma~\ref{lemma:5.2}.
 
 \begin{corollary}\label{lastcor}
Assume that $0<\ep<1$,   $s_j \ge m$  and $a_j = \min(s_j-m, 1-\ep)$, $j=1, \dots ,N$.
Set
$$ \sigma= \min(s_j-a_j) +\sum\limits_{j=1}^N a_j.$$  
Then there is a constant $C$, which depends on $n$,  $\ep$ and the  $s_j$ only,
such that
\begin{equation}\label{bnest:3001}
\normrum{\BB_{N}\vec v}{(\sigma )}^2 \leq C^N R^{N-1} 
\prod\limits_1^N \normrum{v_j}{(s_j)}^2 , 
\end{equation}
for every $N\ge 2$, $R>0$ and $v_1, \dots ,v_N\in C_0^\infty(B(0, R))$.
\end{corollary}

\end{document}